\documentclass[10pt,reqno]{amsart}

\usepackage{a4wide}
\usepackage{dsfont} 
\usepackage{amssymb}
\usepackage{pstricks}
\usepackage[all]{xy}

\newtheorem{proposition}{Proposition}[section]
\newtheorem{theorem}[proposition]{Theorem}

\theoremstyle{remark}
\newtheorem{definition}[proposition]{Definition}
\newtheorem{remark}[proposition]{Remark}

\newcommand{\cst}{\ifmmode\mathrm{C}^*\else{$\mathrm{C}^*$}\fi}
\newcommand{\wst}{\ifmmode\mathrm{C}^*\else{$\mathrm{W}^*$}\fi}

\newcommand{\id}{\mathrm{id}}

\newcommand{\I}{\mathds{1}}

\newcommand{\GG}{\mathbb{G}}
\newcommand{\HH}{\mathbb{H}}

\DeclareMathOperator{\C}{C}
\DeclareMathOperator{\flip}{\Sigma}
\DeclareMathOperator{\Obj}{Obj}

\DeclareMathOperator{\M}{M}
\DeclareMathOperator{\Mor}{Mor}
\DeclareMathOperator{\Aut}{Aut}
\DeclareMathOperator{\B}{B}

\DeclareMathOperator{\Ltwo}{L^2\!\!\;}

\DeclareMathOperator{\RD}{\mathcal{RD}}
\DeclareMathOperator{\CP}{\mathcal{CP}}
\numberwithin{equation}{section}

\begin{document}

\author{Pawe{\l} Kasprzak}
\address{Department of Mathematical Methods in Physics, Faculty of Physics, University of Warsaw, Poland
and Institute of Mathematics of the Polish Academy of Sciences,
ul.~\'Sniadeckich 8, 00--956 Warszawa, Poland}  \email{pawel.kasprzak@fuw.edu.pl}
\thanks{Supported by National Science Centre (NCN) grant no.~2011/01/B/ST1/05011}

\title{Rieffel deformation of tensor functor and braided quantum groups}

\keywords{Rieffel deformation, braiding, quantum groups}
\subjclass[2010]{Primary: 46L89, 46L85, Secondary: 22D35, 58B32}

\begin{abstract}
We apply  Rieffel deformation to  $\C^*$- tensor product  viewed as a functor on the category of  $\C^*$-algebras with an abelian group action. 
In the case of the  Rieffel deformation  of a quantum group with the action by automorphisms the deformed tensor product enables us to view the deformed object as a braided  quantum group.   We construct a bicharacter  for a braided quantum group and the dual braided quantum group. We employ our method to  get a braided quantum Minkowski space. Its description in terms of the deformed space-time coordinates is provided.
\end{abstract}

\maketitle
\section{Introduction} The theory of locally compact quantum groups (LCQG) is well established by now. For the axiomatic formulations of LCQG with the existence of Haar measure postulated  we refer the reader to \cite{KV} or \cite{mnw}. For the theory  with the multiplicative unitary playing central role we refer to \cite{SolWor}. Roughly speaking a quantum group is a pair $\GG = (A,\Delta)$ where $A$ is a $\C^*$-algebra and $\Delta\in\Mor(A,A\otimes A)$. Let us postpone  the discussion of  morphism and multipliers  to Section \ref{cpf} where the respective $\mathcal{C}^*$-category is  introduced and focus on the tensor product  $A\otimes B$ of $\C^*$ - algebras $A$ and $B$. 

Remarkably $A\otimes B$ is not  a uniquely defined $\C^*$-algebra. Between the minimal (spatial) and the  maximal (universal) tensor products $A\otimes_{\textrm{min}} B$ and  $A\otimes_{\textrm{max}} B$ we have  a whole family of   alternative tensor products $A\otimes_\lambda B$  unless  $A\otimes_{\textrm{min}} B = A\otimes_{\textrm{max}} B$. In this paper we restrict our interest to  the spatial  tensor product thus we ignore the subscript ``${\textrm{min}}$'' in the notation writing  $A\otimes B$.  For an enjoyable book with the tensor product of $\C^*$-algebras intelligibly explained we refer to \cite{BrOz}.

 Let    $A$ and $B$   be faithfully  represented on Hilbert spaces  $H$ and $K$ respectively. Then  $A\otimes B$ is defined as the  closed linear span of $\{a\otimes b\in B(H\otimes K):a\in A, b\in B\}$. Remarkably  the resulting $\C^*$-algebra $A\otimes B$ does not depend on the choice of faithful representations. The tensor product construction  lifts to the morphisms level: for $\pi\in\Mor(A_1,B_1)$ and $\sigma\in\Mor(A_2,B_2)$ there exists a unique   morphism  $\pi\otimes\sigma\in\Mor(A_1\otimes B_1,A_2\otimes B_2)$ satisfying $(\pi\otimes\sigma)(a\otimes b) = \pi(a)\otimes\sigma(b)$ for any $a\in A$, $b\in B$.  Summarizing we get the  associative functor   $\otimes:\mathcal{C}^*\times \mathcal{C}^*\rightarrow \mathcal{C}^*$
\[\otimes\circ(\otimes\times\id) = \otimes\circ(\id\times\otimes)\]
together with a pair of distinguished morphisms $\iota^A\in\Mor(A,A\otimes B)$ and  $\iota^B\in\Mor(A,A\otimes B)$
\[\begin{split}
\iota^A: & A \ni a \mapsto a \otimes \I\in\M(A \otimes B)\\
\iota^B: & B\ni b\mapsto  \I\otimes b \in\M(A \otimes B)
\end{split}\] satisfying  $A \otimes B = [\iota^A(A)\cdot\iota^B(B)]$ (for $[X]$ - notation we refer to the last paragraph of this section). 

The tensor product functor has the covariant version  defined in the category  $\mathcal{C}^*_G$ of $\C^*$-algebras acted by a locally compact group $G$ which we discuss in Section \ref{cpf}. An object  of $\mathcal{C}^*_G$ is a pair $(A,\rho^A)$ where $\rho^A:G\rightarrow \Aut(A)$ is an a action of $G$ on $A$ and  the action $\rho^{A\otimes B}_g$ of $G$ on $A\otimes B$   is the diagonal action $\rho^{A\otimes B}_g = \rho^A_g\otimes \rho^B_g$.  

Let $\Gamma$ be an abelian group and $(A,\rho^A)\in\Obj(\mathcal{C}^*_\Gamma)$ and let $\Psi$ be a $2$-cocycle on the Pontryagin dual $\widehat\Gamma$. As we noted in \cite{Kasp} one may define Rieffel deformed $\C^*$-algebra $A^\Psi$ and view it as an object of $\mathcal{C}^*_\Gamma$. In Section \ref{rdf} we view  Rieffel deformation $\RD^\Psi$ as  functor $\RD^\Psi:\mathcal{C}^*_\Gamma\rightarrow\mathcal{C}^*_\Gamma$. Observing the invertibility of $\RD^\Psi$   \[\left(\RD^\Psi\right)^{-1} = \RD^{\overline\Psi}\] we  define  the deformed tensor functor $\otimes_\Psi:\mathcal{C}^*_\Gamma\times \mathcal{C}^*_\Gamma\rightarrow \mathcal{C}^*_\Gamma$  by the formula \[\otimes_\Psi = \RD^\Psi\circ\otimes\circ\left(\RD^{\overline\Psi}\times \RD^{\overline\Psi}\right)\]   and note that $(\mathcal{C}^*_\Gamma,\otimes_\Psi)$ forms a  monoidal category in the sense of  \cite[Definition 9.1.2]{Majid}. In  \cite[Definition 9.2.1]{Majid} S. Majid defines a braided monoidal category, which   is a monoidal category $(\mathcal{C},\otimes)$ with  $\otimes$ and $\otimes^{\textrm{op}}$ being naturally equivalent. We show that applying $\RD^\Psi$ to the flip isomorphisms $\Sigma_{A,B}:A\otimes B\rightarrow B\otimes A$ 
\[\Sigma_{A,B}(a\otimes b) = b\otimes a\] we get a family of isomorphisms   $\sideset{^\Psi}{_{A\otimes B}}\flip:A\otimes_\Psi B\rightarrow B\otimes_\Psi A$   satisfying the  hexagon equation  \cite[Eq. (9.4)]{Majid}. In particular  $(\mathcal{C}^*_\Gamma,\otimes_\Psi,\sideset{^\Psi}{}\flip)$  forms a braided monoidal category. Discussing \cite[Section 9]{Majid} let us note that our passage from $A\otimes B$ to $A\otimes_\Psi B$ is closely related with   S. Majid construction of   noncommutative tensor product for
$H$-comodule algebras over a quasitriangular Hopf algebra $H$ (see \cite[Corollary 9.2.13]{Majid}).

Having  the functor $\otimes_\Psi$ defined we move on in Section \ref{bqg} to the  $\otimes_\Psi$ - braided quantum groups. We observe  that a   quantum group $\GG =(A,\Delta) $ on which $\Gamma$ acts by  $\GG$ - automorphisms gives rise by Rieffel deformation to a braided quantum group $\GG^\Psi = (A^\Psi,\Delta^\Psi)$ with $\Delta^\Psi\in\Mor(A^\Psi,A^\Psi\otimes_\Psi A^\Psi)$ satisfying the coassociativity and cancellation law in the braided sense. Remarkably we get a  candidates for    bicharacter $W^\Psi$ of $\GG^\Psi$ and the dual braided quantum group. The attempt of formulating an analytical  theory of braided quantum groups was undertaken in \cite{Roy} but to the authors knowledge the task is not yet being completed.  For the algebraic counterpart of the notion of a braided quantum group see \cite[Section 9]{Majid}.  
 
In the last section we employ  our method to construct a braided quantum Minkowski space. Its description in terms of the deformed space-time coordinates  $(\hat x_0,\hat x_1,\hat x_2,\hat x_3)$ is provided.  The comultiplication  $\Delta^\Psi$ when applied to the coordinates is shown to  give $\Delta^\Psi(\hat x_i) = \hat x_i\otimes_\Psi \I + \I\otimes_\Psi \hat x_i$, for $i= 0,1,2,3$.
Our construction is consistent with    S. Roy thesis  \cite{Roy} where he observes  that trying to formulate the theory of semidirect product of quantum groups one must allow the counterpart of a normal subgroup to be   a braided quantum group. Thus quantizing the Galilean or Poincar\'e group leads  naturally to braided structures.  In our ongoing project we deform the Poincar\'e group where the role of the normal subgroup is played by the quantum Minkowski space described in this paper. 
A different motivation for considering the  braided structures  was given in  \cite[Section 10]{Majid} where the notion of braided statistics is introduced. As noted by  S. Majid  the fermionic statistics and the supergeometry assigned to it  may be expressed in terms of  braided monoidal category formulated in the context of the  $\mathbb{Z}_2$ - modules and   \cite[Corollary 9.2.13]{Majid}. Thus,  as suggested by S. Majid,  more complicated examples of braided monoidal categories have a potential to be used in physics for a description of more general kind of statistics and the corresponding geometries. 

Some remarks about the notation. For a subset $X$ of a Banach
space $B$, $X^{\textrm{cls}}$ denotes the closed linear span of $X$. Alternatively we shall write $[X] = X^{\textrm {cls}}$.
The (Banach) dual  $A'$ of a $\C^*$-algebra $A$  is
an $A$-bimodule where for $\omega\in A'$ and $b,c\in A$ we
define $b\cdot\omega\cdot c$ by the formula:
\[(b\cdot\omega\cdot c)(a)=\omega(cab)\]
for any $a\in A$. The $\C^*$-algebra  of  multipliers of $A$ is denoted by $\M(A)$. $\M(A)$ is equipped with the strict topology - for the discussion of natural $\C^*$-topologies we refer to \cite{woraffun}. The group $\C^*$ - algebra of a locally compact group $G$ will be denoted by $\C^*(G)$; the representatives  of $G$ inside $\M(\C^*(G))$ will be denoted by   $\lambda_g\in\M(\C^*(G))$. 
\section{$\mathcal{C}^*_\Gamma$-category and the crossed product functor}\label{cpf}
Since the results of this paper are naturally expressed  in the   category theory terms let us first introduce the category $\mathcal{C}^*$ with $\C^*$-algebras being its objects and Woronowicz morphisms. For $A,B\in\Obj(\mathcal{C}^*)$, a   morphism  $\pi:A\rightarrow B$  is a $*$-homomorphism $\pi:A\rightarrow \M(B)$ which is non-degenerate $\overline{\pi(A)B}^{\|\cdot\|} = B$.  One may introduce the unique extension  $\bar{\pi}:\M(A)\rightarrow\M(B)$ of $\pi$
satisfying $\bar\pi(a_1)\pi(a_2)b = \pi(a_1a_2)b$ for all $a_1,a_2\in A$ and $b\in B$. Denoting the extension by the same symbol $\pi$  we have a well defined morphisms composition: $\sigma\circ\pi\in\Mor(A,C)$ for any $\pi\in\Mor(A,B)$ and $\sigma\in\Mor(B,C)$.

Let $\Gamma$ be an abelian locally compact group. We use the additive notation $\gamma +\gamma'\in \Gamma$ for any $\gamma,\gamma'\in\Gamma$. The category of $\Gamma$-$\C^*$-algebras is denoted by $\mathcal{C}^*_\Gamma$. We write $(A,\rho^A) \in \Obj(\mathcal{C}^*_\Gamma)$ where $A$ is a $\C^*$ algebra  and  $\rho^A:\Gamma\rightarrow\Aut(A)$ is a continuous  action, i.e. $\rho^A_{\gamma+\gamma'} = \rho^A_{\gamma}\circ\rho^A_{\gamma'}$ and the map \[
\Gamma\ni\gamma\mapsto\rho_\gamma(a)\in A\] is norm continuous for any $a\in A$. A morphism $\pi:(A,\rho^A)\rightarrow (B,\rho^B) $ is a morphism $\pi\in\Mor(A,B)$   which is covariant  \[\rho^B_\gamma\circ\pi = \pi\circ\rho_\gamma^A.\] We  write $\Mor_\Gamma$ and   $\pi\in\Mor_\Gamma(A,B)$ avoiding  to put $\rho^A,\rho^B$ under the $\Mor_\Gamma$-symbol.  

An element $(A,\rho^A)\in\Obj(\mathcal{C}^*_\Gamma)$ gives rise to the crossed product $\C^*$-algebra $\Gamma\ltimes A\in \Obj(\mathcal{C}^*)$. For the crossed product theory we refer to \cite{Wil} recalling only the essential properties of $\Gamma\ltimes A$. Remarkably  $A$ and  $\C^*(\Gamma)$   embed  into $\M(\Gamma\ltimes A)$ via \[\begin{split} \iota^A&\in\Mor(A,\Gamma\ltimes A)\\\iota^{\C^*(\Gamma)}&\in\Mor(\C^*(\Gamma),\Gamma\ltimes A) \end{split}\] and identifying $A,\C^*(\Gamma)$ with the subalgebras of $\M(\Gamma\ltimes A)$   we have  \[\Gamma\ltimes A = [\C^*(\Gamma)\cdot A]\] 
\begin{definition}
Let $(A,\rho^A)\in\Obj(\mathcal{C}^*_\Gamma)$ and let $H$ be a Hilbert space. Let $\pi:A\rightarrow B(H)$ be a non-degenerate representation of $A$ on $H$ and $\theta:\Gamma\rightarrow B(H)$   a unitary strongly continuous representation of $\Gamma$ on $H$. We say that $(\theta,\pi)$ is a covariant representation of $(A,\rho^A)$ on $H$ if for any $a\in A$ and $\gamma\in\Gamma$ we have 
\[\pi(\rho^A_\gamma(a)) = \theta_\gamma\pi(a)\theta_\gamma^*.\]
\end{definition} Let $(\theta,\pi)$ be a covariant representation of $(A,\rho^A)$ on a Hilbert space $H$. The unique extension of   $\theta:\Gamma\rightarrow B(H)$ to the $\C^*$ - representation satisfying \[\M(\C^*(\Gamma))\ni\lambda_\gamma\mapsto\theta_\gamma\in\B(H)\]  will be denoted by the same symbol $\theta:\C^*(\Gamma)\rightarrow\B(H)$.  The universal property of $\Gamma\ltimes A$ is formulated as a 1-1 correspondence between covariant representations of $(A,\rho^A)$ and   non-degenerate representation  of $\Gamma\ltimes A$ where the representation assigned to $(\theta,\pi)$  denoted by $ \theta\ltimes\pi$ is characterized as the unique representation satisfying
\[
\begin{split}
(\theta\ltimes\pi)\circ\iota^A &= \pi\circ\iota^A \\
(\theta\ltimes\pi)\circ\iota^{\C^*(\Gamma)}  &= \theta
\end{split}
\]
The universal property of $\Gamma\ltimes A$ for covariant representations of $(A,\rho^A)\in\Obj(\mathcal{C}^*_\Gamma)$ on a Hilbert space $H$ has its counterpart for covariant representations on $\C^*$ - algebras. 
 \begin{definition}
Let $(A,\rho^A)\in\Obj(\mathcal{C}^*_\Gamma)$, $C\in\Obj(\mathcal{C}^*)$, $\pi\in\Mor(A,C)$ and $\theta:\Gamma\rightarrow\M(C)$ a    unitary strictly continuous   representation. We say that $(\theta,\pi)$ is a covariant representation  of $(A,\rho^A)$ on $C$ if for any $a\in A$ and $\gamma\in\Gamma$ we have 
\[\pi(\rho^A_\gamma(a)) = \theta_\gamma\pi(a)\theta_\gamma^*.\]
\end{definition}
For a covariant representation $(\theta,\pi)$   there exists  $\theta\ltimes\pi\in\Mor(\Gamma\ltimes A, C)$  uniquely characterized  by \[
\begin{split}
(\theta\ltimes\pi)\circ\iota^A &= \pi\circ\iota^A \\
(\theta\ltimes\pi)\circ\iota^{\C^*(\Gamma)}  &= \theta
\end{split}
\] Let $\widehat\Gamma$ be the  Pontryagin  dual of $\Gamma$ with the duality 
\[\widehat\Gamma\times\Gamma\ni(\widehat\gamma,\gamma)\mapsto \langle\widehat\gamma,\gamma\rangle\in\mathbb{T}^1\] and let $\widehat\gamma\in\widehat{\Gamma}$. Noting that $\Gamma\ni\gamma\mapsto \langle\widehat\gamma,\gamma\rangle\lambda_\gamma\in\M(\Gamma\ltimes A)$ is a representation such that the pair  $(\langle\widehat\gamma,\cdot\rangle\lambda,\iota^A)$ is a covariant representation and using  universality of $\Gamma\ltimes A$ we get a morphism $\widehat\rho_{\widehat\gamma}\in\Mor(\Gamma\ltimes A,\Gamma\ltimes A)$. Since $\rho_{\widehat\gamma+\widehat\gamma'} = \widehat\rho_{\widehat\gamma}\circ\widehat \rho_{\widehat\gamma'}$ we may see that $\widehat\rho_{\widehat\gamma}\in\Aut(\Gamma\ltimes A)$ and  we get  the famous dual  action $\widehat\rho$ of $\widehat\Gamma$ on $\Gamma\ltimes A$.  

Let $\pi\in\Mor_\Gamma(A,B)$. Again, applying the universal property of $\Gamma\ltimes A$ in the context of the covariant representation $(\lambda,\pi)$ we get $\lambda\ltimes\pi\in\Mor(\Gamma\ltimes A,\Gamma\ltimes B)$.  Its  covariance with respect to the dual actions is clear and enables us to view 
the crossed product  as a functor $\CP: \mathcal{C}^*_\Gamma\rightarrow \mathcal{C}^*_{\widehat{\Gamma}}$  \[
\begin{split}
\CP((A,\rho^A)) &= (\Gamma\ltimes A,\hat\rho)\\
\CP(\pi) & = \lambda\ltimes\pi
\end{split}\] 
 
Remarkably we have a characterization of    $(D,\hat\rho^D)\in\Obj(\mathcal{C}^*_{\widehat\Gamma})$ that are of the crossed product form. Note that for $(\Gamma\ltimes A,\hat\rho)\in\Obj(\mathcal{C}^*_{\widehat\Gamma})$ we have a representation $\lambda:\Gamma\rightarrow\M(\Gamma\ltimes A)$ and 
\[\hat\rho_{\hat\gamma}(\lambda_\gamma) = \langle\hat\gamma,\gamma\rangle\lambda_\gamma.\] As was shown by Landstad \cite{lan}, a $\widehat\Gamma$-$\C^*$-algebra  $(D,\hat\rho^D)\in\Obj(\mathcal{C}^*_{\widehat\Gamma})	$ that is equipped with a representation 	$\lambda:\Gamma\rightarrow\M(D)$		satisfying \[\hat\rho^D_{\hat\gamma}(\lambda_\gamma) = \langle\hat\gamma,\gamma\rangle\lambda_\gamma \] may be identified with a crossed product of a certain (essentially unique)  $(A,\rho^A)$. Extending $\lambda$ to an injective morphism $\lambda\in\Mor(\C^*(\Gamma),D)$, we characterize $A\subset \M(D)$ as the $\Gamma$-$\C^*$-algebra satisfying Landstad conditions

\begin{equation}\label{lancod} A=\left\{d\in M(D)\left|\begin{array}{l}1.\,\,\hat\rho^D_{\widehat\gamma}(d)=d\\
2.\,\,\mbox{The map }\Gamma\ni\gamma\mapsto\lambda_\gamma d\lambda_\gamma^*\in \M(D)\\
\mbox{\,\,\,\,\,\, is norm-continuous}\\
3.\,\, xdy\in D \mbox{ for any }x,y\in \C^*(\Gamma)
\end{array}\right.\right\}\end{equation}

where $\rho^A$ is implemented by $\lambda$
\[\rho^A_\gamma(a) = \lambda_\gamma a \lambda_\gamma^*.\] 
A triple $(D,\hat\rho^D,\lambda)$ will be called a $\Gamma$-product and $A\subset \M(D)$ a Landstad algebra of the corresponding  $\Gamma$-product.
\begin{remark}\label{repcp}
Let $H$ be a Hilbert space and $A\subset\B(H)$. Then there exists a faithful representation  $\Gamma\ltimes A$  on $\Ltwo(\Gamma)\otimes H$ that corresponds to  a unique covariant representation  $(\theta,\pi)$. Identifying $\Ltwo(\Gamma)\otimes H$ with the Hilbert space of square integrable  maps from $\Gamma$ to $H$  \[\Ltwo(\Gamma)\otimes H \cong \Ltwo(\Gamma,H)\] we have \[\begin{split}(\theta_\gamma x)(\gamma') &= x(\gamma'+\gamma)\\(\pi(a)x)(\gamma)  &= \rho^A_\gamma(a)x(\gamma)\end{split}\]  for any $x\in \Ltwo(\Gamma,H)$ and $\gamma,\gamma'\in\Gamma$. 
\end{remark}
\section{Rieffel deformation functor}\label{rdf}
Rieffel deformation as formulated in this section was introduced in \cite{Kasp} - here we rephrase it in   categorical terms. For the original approach developed by M. Rieffel we refer to \cite{Rf1}. 

Let $\Psi:\widehat\Gamma\times\widehat\Gamma\rightarrow\mathbb{T}^1$ be a continuous  $2$-cocycle on $\widehat\Gamma$
\begin{itemize}
\item[(i)] $\Psi(e,\widehat\gamma) =  \Psi(\widehat\gamma, e) = 1 \textrm{ for all } \widehat\gamma\in\widehat\Gamma$
\item[(ii)] $\Psi(\widehat\gamma_1,\widehat\gamma_2+\widehat\gamma_3)\Psi(\widehat\gamma_2,\widehat\gamma_3)= \Psi(\widehat\gamma_1+\widehat\gamma_2,\widehat\gamma_3)\Psi(\widehat\gamma_1,\widehat\gamma_2) \textrm{ for all } \widehat\gamma_1,\widehat\gamma_2,\widehat\gamma_3\in\widehat\Gamma $
\end{itemize}
Note that the set $\mathcal{COC}_2(\widehat\Gamma)$ of $2$-cocyles on $\widehat\Gamma$ forms  a multiplicative group and $\Psi^{-1} = \overline\Psi$. 
For any $\widehat\gamma\in\widehat\Gamma$ we define an auxiliary function $\Psi_{\widehat\gamma}:\widehat\Gamma\rightarrow\mathbb{T}$ where 
\[\Psi_{\widehat\gamma}(\widehat\gamma') = \Psi(\widehat\gamma', \widehat\gamma) \] - its role in the theory will be explained later. 

In this section we shall write $(A,\rho)$ ignoring $A$ in  $\rho^A$. For $(A,\rho)\in\Obj(\mathcal{C}^*_\Gamma)$ we form the crossed product $\Gamma\ltimes A$. Identifying $\C^*(\Gamma)$ with $\C_0(\widehat\Gamma)$ we get a unitary family $U_{\widehat\gamma} = \iota^{\C^*(\Gamma)}(\Psi_{\widehat\gamma})\in\M(\Gamma\ltimes A)$. 
For any $\widehat\gamma\in\widehat\Gamma$ we define $\widehat\rho_{\widehat\gamma}
^\Psi\in\Aut(\Gamma\ltimes A)$ by the formula
\begin{equation}\label{defdulact}\widehat\rho_{\widehat\gamma}
^\Psi(b) = U_{\widehat\gamma}^*\widehat\rho_{\widehat\gamma}(b)U_{\widehat\gamma}\end{equation}
As was shown in \cite[ Theorem 3.1]{Kasp},   $\widehat\rho_{\widehat\gamma}
^\Psi$ defines a continuous action of $\widehat\Gamma$ on  $\Gamma\ltimes A$ and the triple  
$(\Gamma\ltimes A, \widehat\rho_{\widehat\gamma}^\Psi,\lambda)$ is a $\Gamma$ - product. The Landstad algebra  of this $\Gamma$ - triple is called a Rieffel deformation of $A$ and denoted $A^\Psi$. Remarkably  $A^\Psi$ being  the Landstad algebra of $(\Gamma\ltimes A, \widehat\rho_{\widehat\gamma}^\Psi,\lambda)$ is equipped with $\Gamma$-action that is implemented by $\lambda$:
\[A\ni a\mapsto \lambda_\gamma a \lambda_\gamma^*\in A\] for any $a\in A^\Psi$ which we denote by $\rho:\Gamma\rightarrow\Aut(A^\Psi)$. Note that  $A^\Psi$ being the Landstad algebra of the triple $(\Gamma\ltimes A, \widehat\rho_{\widehat\gamma}^\Psi,\lambda)$ satisfies  $\Gamma\ltimes A^\Psi = \Gamma\ltimes A$.  

Let $\pi\in\Mor_\Gamma(A,B)$ and $\lambda\ltimes\pi\in\Mor_{\widehat\Gamma}(\Gamma\ltimes A,\Gamma\ltimes B)$. It was shown in \cite{Kasp} that $(\lambda\ltimes \pi)(A^\Psi)\subset\M(B^\Psi)$ and $\pi$ gives rise to   a  morphism $\pi^\Psi\in\Mor_\Gamma(A^\Psi,B^\Psi)$ where $\pi^\Psi = \lambda\ltimes\pi|_{A^\Psi}$. Since for $\pi\in\Mor_\Gamma(A,B)$ and $\sigma\in\Mor_\Gamma(B,C)$ we have 
\[\begin{split}\lambda\ltimes(\sigma\circ\pi)&=  (\lambda\ltimes\sigma)\circ(\lambda\ltimes\pi)\end{split}\] 
we get  $(\sigma\circ\pi)^\Psi = \sigma^\Psi\circ\pi^\Psi$. 
\begin{definition}
Let $\Gamma$ be an abelian group. For any $\Psi\in \mathcal{COC}_2(\widehat\Gamma)$ we define a functor 
\[\RD^\Psi:\mathcal{C}^*_\Gamma\rightarrow\mathcal{C}^*_\Gamma\] such that adopting the above notation we have 
\[\begin{split} \RD^\Psi(A,\rho) &= (A^\Psi,\rho)\\
\RD^{\Psi}(\pi) &= \pi^\Psi
\end{split}\] which we call a Rieffel deformation functor associated to $\Psi$.
\end{definition}
For any category $\mathcal{C}$, a functor $\mathcal{R}:\mathcal{C}\rightarrow \mathcal{C}$ is  called an endofunctor. An endofunctor $\mathcal{R}$ having an inverse is  an autofunctor. The class of autofunctors of a category $\mathcal{C}$ is denoted by $\mathcal{AFUN}(\mathcal{C})$. 
\begin{remark}   Note that  the set of autofunctors $\{\RD^\Psi:\Psi\in \mathcal{COC}_2(\widehat\Gamma)\}$ is closed under the composition and $ \RD^{\Psi\Phi} = \RD^{\Psi}\circ\RD^{\Phi}$. Since $\Psi^{-1} = \overline{\Psi}$ we have 
\begin{equation}\label{invrd}(\RD^{\Psi})^{-1} = \RD^{\overline{\Psi}}
\end{equation}
Let us also mention that $\RD^{\Psi}$ is exact (see Proposition 2.9 of \cite{Kasp}) . In particular $\RD^\Psi(\pi)$  is injective if (and only if by invertibility of $\RD^\Psi)$  $\pi$ is injective.
\end{remark}
\subsection{Rieffel deformation of tensor product}
Let $(A,\rho^A),(B,\rho^B)\in\Obj(\mathcal{C}^*_\Gamma)$. Defining the action $\rho^{A\otimes B}$ of $\Gamma$ on $A\otimes B$  \[\rho^{A\otimes B}_\gamma(a\otimes b) = \rho^A_\gamma(a)\otimes \rho^B_\gamma(b)\] we get the tensor functor $\otimes:\mathcal{C}^*_\Gamma\times\mathcal{C}^*_\Gamma\rightarrow \mathcal{C}^*_\Gamma$. Clearly the functor $\otimes$ is associative: \[\otimes\circ(\otimes\times\id) = \otimes\circ(\id\times\otimes)\] Moreover for $(A,\rho^A),(B,\rho^B)\in\Obj(\mathcal{C}^*_\Gamma)$ we have a pair of injective morphisms \[\begin{split}\iota^A\in\Mor_\Gamma(A,A\otimes B):&\,\,\iota^A(a) = a\otimes 1\\ \iota^B\in\Mor_\Gamma(B,A\otimes B):&\,\,\iota^B(b) = 1\otimes b
\end{split}
\] satisfying $A\otimes B = [\iota^A(A)\cdot\iota^B(B)]$.
\begin{definition}
Let $\Psi\in\mathcal{COC}_2(\widehat\Gamma)$. We define the functor $\otimes_\Psi:\mathcal{C}^*_\Gamma\times\mathcal{C}^*_\Gamma\rightarrow \mathcal{C}^*_\Gamma$ \[\otimes_\Psi = \RD^\Psi\circ\otimes\circ (\RD^{\bar\Psi}\times\RD^{\bar\Psi})\]
\end{definition}
Note that 
\begin{equation}\label{inter}A^\Psi\otimes_\Psi B^\Psi = (A\otimes B)^\Psi\end{equation}
 Since $\iota^A\in\Mor_\Gamma(A,A\otimes B)$ we immediately see that $\RD^\Psi(\iota^A)\in\Mor_\Gamma(A^\Psi,A^\Psi\otimes_\Psi B^\Psi)$ is an injective morphism. Similarly, for $\iota^B$ we get injecitve $\RD^\Psi(\iota^B)\in\Mor_\Gamma(B^\Psi,A^\Psi\otimes_\Psi B^\Psi)$. We shall use a notation
\begin{equation}\label{sugnot}\begin{split}
\RD^\Psi(\iota^A)(a) &=a\otimes_\Psi I\\
\RD^\Psi(\iota^B)(b) &=\I\otimes_\Psi b
\end{split}
\end{equation}
for $a\in A^\Psi$ and $b\in B^\Psi$.

\begin{proposition}
Adopting the above notation,  the functor $\otimes_\Psi$ is associative $\otimes_\Psi\circ (\otimes_\Psi\times\id) = \otimes_\Psi\circ (\id\times \otimes_\Psi)$.
Moreover $A^\Psi\otimes_\Psi B^\Psi = [(A^\Psi\otimes_\Psi\I)\cdot (\I\otimes_\Psi B^\Psi)]$.
\end{proposition}
\begin{proof}
 Let $A,B,C \in\mathcal{C}^*_\Gamma$. We have 
\[\begin{split}(A^\Psi\otimes_\Psi B^\Psi)\otimes_\Psi C^\Psi&=(A\otimes B)^\Psi\otimes_\Psi C^\Psi\\
&=((A\otimes B) \otimes C)^\Psi\\&=(A\otimes(B\otimes C))^\Psi=A^\Psi\otimes_\Psi(B^\Psi\otimes_\Psi C^\Psi)
\end{split}\]
Since $\RD^\Psi$ is an invertible functor and all $\C^*$-algebras are of the form $A^\Psi$ we get associativity of $\otimes_\Psi$. 
In order to prove the second part of the proposition 
we use Lemma 3.4 of \cite{kasp1} in the context of $\iota^A(A),\iota^B(B)\subset\M(A\otimes B)$ getting   $A^\Psi\otimes_\Psi B^\Psi = [(A^\Psi\otimes_\Psi\I)\cdot (\I\otimes_\Psi B)]$.
\end{proof}
\begin{remark}\label{defmorph}
Applying $\otimes_\Psi$ -functor to  $\pi_1\in\Mor_\Gamma(A_1,B_1)$ and $\pi_2\in\Mor_\Gamma(A_2,B_2)$ we get   $\pi_1\otimes_\Psi\pi_2\in\Mor_\Gamma(A_1\otimes_\Psi A_2,B_1\otimes_\Psi B_2)$. In particular for $C\in\Obj(\mathcal{C}^*_\Gamma)$ and $\pi\in\Mor_\Gamma(A,B)$   we have  $\pi\otimes_\Psi\id\in\Mor_\Gamma(A\otimes_\Psi C,B\otimes_\Psi C)$. 
\end{remark}

Let $A,B\in\Obj(\mathcal{C}^*_\Gamma)$ and let $\Sigma_{A, B}:A\otimes B\rightarrow B\otimes A$ be the flip isomorphism
\[\Sigma_{A, B}(a\otimes b) = b\otimes a\]
Noting that $\Sigma_{A,B}\in\Mor_\Gamma(A\otimes B,B\otimes A)$ 
\[\Sigma_{A,B}\circ\rho^{A\otimes B}_{\gamma} = \rho^{B\otimes A}_{\gamma}\circ\Sigma_{B,A}\] and applying Rieffel functor we get the family of isomorphisms
\[\Sigma^\Psi_{A,B} = \RD^\Psi(\Sigma_{A,B})\in\Mor(A^\Psi\otimes_\Psi B^\Psi,B^\Psi\otimes_\Psi A^\Psi)\] 
Defining $\sideset{^\Psi}{_{A,B}}\flip = \sideset{}{^\Psi_{A^{\overline{\Psi}},B^{\overline{\Psi}}}}\flip\in \Mor(A\otimes_\Psi B,B\otimes_\Psi A)$ 
one may easily check that the  hexagon equations \cite[Eq. (9.4)]{Majid}
\[\begin{split} \Sigma_{A\otimes B, C}&=(\Sigma_{A,C}\otimes\id_B)\circ(\id_A\otimes \Sigma_{B,C}) \\
 \Sigma_{A, B\otimes C}&= (\id_B\otimes\Sigma_{A,C})\circ(\Sigma_{A,B}\otimes\id_C)
 \end{split}
\]  still hold  after Rieffel deformation
\[\begin{split} \sideset{^\Psi}{_{A\otimes_\Psi B,C}}\flip&=(\sideset{^\Psi}{_{A,C}}\flip \otimes_\Psi\id_B)\circ(\id_A\otimes_\Psi \sideset{^\Psi}{_{ B,C}}\flip) \\\sideset{^\Psi}{_{A, B\otimes_\Psi C}}\flip&=(\id_B\otimes_\Psi \sideset{^\Psi}{_{ A,C}}\flip)\circ(\sideset{^\Psi}{_{A,B}}\flip \otimes_\Psi\id_C)\end{split}
\]
Thus we may summarize the results of this section in the following theorem where we adopt \cite[Definition 9.2.1]{Majid} of braided monoidal category. 
\begin{theorem}
Adopting the notation of this section, the triple $(\mathcal{C}^*_\Gamma,\otimes_\Psi,\sideset{^\Psi}{}\flip)$ is a braided monoidal category. 
\end{theorem}
\section{Rieffel deformation and braided quantum groups}\label{bqg}
In this section we  apply Rieffel deformation $\RD^\Psi$ to a quantum group $\GG$ with $\Gamma$ acting by automorphisms of $\GG$. We  adopt   \cite[Definition 2.3]{SolWor} and use multiplicative unitary $W$ as the main object of the quantum groups theory. In particular Haar weights are not  discussed in this section. For the theory of  manageable multiplicative unitaries   we refer to \cite{W5} and for the theory of modular multiplicative unitaries   we refer to \cite{W5sol}.
\begin{definition}\label{defsw}
Let $A$ be a $\C^*$
-algebra and $\Delta\in\Mor(A, A\otimes A)$. We say that a  pair $\GG = (A, \Delta)$
is a quantum group if there exists a Hilbert space $H$ and a modular multiplicative unitary $W\in\B(H\otimes H)$  such that $(A, \Delta)$ is isomorphic
to the $\C^*$-algebra with comultiplication associated to $W$. In
such a case we shall say that $W$ is a modular multiplicative unitary giving rise to the quantum
group $\GG$.
\end{definition}
Let $\GG$ be a quantum group with a modular multiplicative unitary $W\in\B(H\otimes H)$. When convenient  we write $A = \C_0(\GG)$ and $H = \Ltwo(\GG)$ ignoring the fact hat $H$ does not have to correspond to the GNS Hilbert space  for the Haar weight of $\GG$. 
 A quantum group $\GG$ has a dual quantum group $\widehat\GG$. The modular multiplicative   unitary of $\widehat\GG$ is given by $\widehat W = \Sigma W^* \Sigma$ where $\Sigma:H\otimes H\rightarrow H\otimes H$ is the flip  operator: $\Sigma (x\otimes y) = y\otimes x$ for $x,y\in H$.  Note that $\widehat{\widehat\GG} = \GG$. 
 
Let $\HH$ and $\GG$ be locally compact quantum groups. In order to introduce a concept of a quantum group homomorphism  $\pi:\HH\rightarrow \GG$ we need  the universal objects  $\HH^u = (\C_0^u(\HH),\Delta_{\HH}^u)$ and $\GG^u = (\C_0^u(\GG),\Delta_{\GG}^u)$ assigned to $\HH$ and $\GG$ respectively. For the definition of the universal version of a quantum group given by a  multiplicative unitary $W$  we refer to \cite[Definition 5.1]{SolWor}.  Let us emphasize that the universal version of a quantum group is not a quantum group in the sense of Definition  \ref{defsw}. 
The concept of quantum groups homomorphism  in the framework of Definition \ref{defsw}  was developed in \cite{MRW}.
\begin{definition}\label{hom} Let $\HH$ and $\GG$ be quantum groups and $\pi\in\Mor(\C_0^u(\GG),\C_0^u(\HH))$.  
We say that $\pi$ is a strong quantum homomorphism and write $\pi:\HH\rightarrow\GG$ if   \[
\Delta_{\HH}^u\circ\pi = (\pi\otimes\pi)\circ\Delta_{\GG}^u
\] When convenient    we refer to a strong quantum homomorphism $\pi:\HH\rightarrow\GG$ as homomorphism ignoring the strong and quantum adjectives. 
\end{definition}
\begin{remark}\label{redlev} Let $\HH = (\C_0(\HH),\Delta_\HH)$ and $\GG = (\C_0(\GG),\Delta_\GG)$ be quantum groups. A homomorphism   $\pi:\HH\rightarrow \GG$ has a  dual  homomorphism $\widehat\pi:\widehat\GG\rightarrow\widehat\HH$ assigned to it and we have $\widehat{\widehat\pi} = \pi$. For  a quantum group  $\mathbb{K}$ and a homomorphism $\sigma:\GG\rightarrow\mathbb{K}$ we have $\widehat{\sigma\circ\pi}=\widehat\pi\circ\widehat\sigma$. In particular $\pi:\HH\rightarrow\GG$ is invertible  if and only if $\widehat\pi$ is invertible. An isomorphism is homomorphism which is  invertible and an isomorphism  $\pi:\GG\rightarrow\GG$ is called an automorphism.
The set of automorphism is denoted by $\Aut(\GG)$.

Let $\pi :\HH\rightarrow \GG$ be an isomorphism. As was shown in \cite[Theorem 1.10]{DKSS} ,  $\pi$ descends to  the $\C^*$-isomorphism  
$\pi_r:\C_0(\GG)\rightarrow \C_0(\HH)$  such that
\[\Delta_\HH\circ\pi_r = (\pi_r\otimes\pi_r)\circ\Delta_{\GG}.\] In particular discussing the automorphism group $\Aut(\GG)$, the universal version $\GG^u$ is irrelevant. 
Let $\pi\in\Mor(\C_0(\GG),\C_0(\GG))$ be an automorphism of $\GG$ and $W$ the  modular multiplicative unitary viewed as $W\in\M(\C_0(\widehat\GG)\otimes \C_0(\GG))$. Then the dual  automorphism $\widehat\pi:\C_0(\widehat\GG)\rightarrow \C_0(\widehat\GG)$ is uniquely characterized by the equality
\begin{equation}\label{dualityrel}(\widehat\pi\otimes\id)W = (\id\otimes \pi)W\end{equation}
\end{remark}
\begin{remark}\label{rem2}
Let $\GG = (A,\Delta)$ be a quantum group. Then $\Aut(\GG)$ equipped with homomorphism composition  forms a group. Equipping   $\Aut(\GG)$ with the weakest topology such that for all $a\in A$ the mappings 
\[\Aut(\GG)\ni \pi\rightarrow \pi(a)\in A\] are norm-continuous, we view $\Aut(\GG)$  as a topological group.  Noting that the  automorphisms duality $\Aut(\GG)\ni\pi\mapsto\widehat\pi\ni\Aut(\widehat\GG)$ is composition reversing $\widehat{\sigma\circ\pi} = \widehat{\pi}\circ\widehat{\sigma}$ we see that   
the map \begin{equation}\label{autom}\Aut(\GG)\ni\pi\mapsto\widehat\pi\in\Aut(\widehat\GG)^{\textrm{op}}\end{equation} is a groups   isomorphism.  
The relation $(\widehat\pi\otimes\id)W = (\id\otimes\pi)W$ enables us to prove that the map \eqref{autom} is in fact  a  topological groups isomorphism. 
\end{remark}
\begin{definition}
Let $ \GG = (A,\Delta)$ be a quantum group, $\Gamma$ an abelian group and $(A,\rho^A)\in\Obj(\mathcal{C}^*_\Gamma)$.  If   $ \rho^A_\gamma\in\Aut(\GG)$ then we say that $\rho^A$ is an action of $\Gamma$  on $\GG$  by automorphisms.  
\end{definition}
Let $\rho^A$ be an action of $\Gamma$  on $A$ and $\GG = (A,\Delta)$. Then $\rho^A$ is an action by automorphism  if and only if $\Delta\in\Mor_{\Gamma}(A,A\otimes A)$. Clearly $\Gamma\ni\gamma\mapsto\rho_\gamma\in\Aut(\GG)$ is a continuous homomorphism. For $\gamma\in\Gamma$ we may consider the dual morphism $\widehat{\rho^A}_\gamma\in\Aut(\widehat\GG)$ (which should not be confused with the dual action $\widehat\rho_{\widehat\gamma}$). Since $\Gamma$ is abelian the map $\Gamma\ni\gamma\mapsto \widehat{\rho^{A}}_\gamma\in\Aut(\C_0(\widehat\GG))$ is an action of $\Gamma$ on $\widehat\GG$ by automorphisms. 
In what follows $\widehat{\rho^{A}}$ will be denoted by $\rho^{\widehat{A}}$.
\subsection{Rieffel deformation of a quantum group with an action of $\Gamma$ by automorphisms}
Let $\GG = (A,\Delta)$ be a quantum group with an action $\rho$ of $\Gamma$ by automorphisms  
and let $\Psi$ be   a $2$-cocyle on $\widehat\Gamma$. Applying  $\RD^\Psi$ to $A$ we get a $\C^*$-algebra $A^\Psi$. In this section we show that $A^\Psi$ is equipped with a morphism  $\Delta^\Psi\in\Mor_\Gamma(A^\Psi,A^\Psi\otimes_\Psi A^\Psi)$ which is coassociative in the sense of the deformed tensor functor $\otimes_\Psi$ and satisfies  cancellation law; we shall refer to $\Delta^\Psi$ as a braided comultiplication. Assuming $\Psi$ to be a bicharacter (which we do in what follows)
\[\begin{split}\Psi(\widehat\gamma_1,\widehat\gamma_2 + \widehat\gamma_3) &= \Psi(\widehat\gamma_1,\widehat\gamma_2) \Psi(\widehat\gamma_1,\widehat\gamma_3)\\
\Psi(\widehat\gamma_1+\widehat\gamma_2,\widehat\gamma_3) &= \Psi(\widehat\gamma_1,\widehat\gamma_3) \Psi(\widehat\gamma_2,\widehat\gamma_3)
\end{split}
\]  we deform the multiplicative unitary $W^\Psi$ and introduce   a certain deformation of $\widehat{\GG}$ that is a good candidate for a  dual of $\GG^\Psi$. 

To a bicharacter $\Psi$ we assign a group homomorphism $\tilde\Psi:\widehat\Gamma\rightarrow \Gamma$ uniquely characterized by
\[\langle\widehat\gamma',\tilde\Psi(\widehat\gamma)\rangle = \Psi(\widehat\gamma',\widehat\gamma).\]
The unitary family $U_{\widehat\gamma}$ entering \eqref{defdulact} is in this case given by 
$U_{\widehat\gamma} = \lambda_{\tilde\Psi(\widehat\gamma)}$ and we have 
\[\widehat\rho^\Psi_{\widehat\gamma}(b) = \lambda^*_{\tilde\Psi(\widehat\gamma)}\widehat\rho_{\widehat\gamma}(b) \lambda_{\tilde\Psi(\widehat\gamma)}.\]
 Applying  Rieffel functor $\RD^\Psi$ to $(A,\rho)$ we get
 $(A^\Psi,\rho)$ and applying it to  $\Delta\in\Mor_\Gamma(A,A\otimes A)$  we get   $\Delta^\Psi\in\Mor_\Gamma(A^\Psi,(A\otimes A)^\Psi)$. Using \eqref{inter} 
we obtain \[ \Delta^\Psi\in\Mor_\Gamma(A^\Psi,A^\Psi\otimes_\Psi A^\Psi) \] In the formulation of the next proposition we use notation introduced in Remark \ref{defmorph}.
 \begin{proposition} Let 
 $\Delta^\Psi\in\Mor(A^\Psi,A^\Psi\otimes_\Psi A^\Psi)$ be the morphisms introduced above. Then
\begin{equation}\label{coassoc}(\Delta^\Psi\otimes_\Psi\id)\circ\Delta^\Psi = (\id\otimes_\Psi\Delta^\Psi)\circ\Delta^\Psi\end{equation} \end{proposition} 
\begin{proof}
Identity \eqref{coassoc} is obtained by application of the Rieffel functor $\RD^\Psi$ to $(\Delta\otimes\id),(\id\otimes\Delta)\in\Mor_\Gamma(A,A^{3\otimes})$ and $\Delta\in\Mor_\Gamma(A,A\otimes A)$:
\[\begin{split}(\Delta^\Psi\otimes_\Psi\id)\circ\Delta^\Psi &=\RD^\Psi((\Delta\otimes\id)\circ\Delta)\\& = \RD^\Psi((\id\otimes\Delta)\circ\Delta) = (\id\otimes_\Psi\Delta^\Psi)\circ\Delta^\Psi\end{split}\] 
\end{proof}

\begin{proposition}
Adopting the above notation we have 
\[[(A^\Psi\otimes_\Psi\I)\cdot \Delta^\Psi(A^\Psi)] = [(\I\otimes_\Psi A^\Psi)\cdot \Delta^\Psi(A^\Psi)]= A^\Psi\otimes_\Psi A^\Psi \]
\end{proposition} 
\begin{proof}  The proof is a straightforward application of  Lemma 3.4 of \cite{Kasp} to $A\otimes 1, \Delta(A)\subset A\otimes A$ and $1\otimes A, \Delta(A)\subset A\otimes A$ accordingly.
\end{proof}

Let $\rho^A$ be an action of $\Gamma$ on $\GG = (A,\Delta)$ by automorphisms. Using Remark \ref{rem2} we get the action $\rho^{\widehat A}$  of $\Gamma$ on $\widehat\GG$ by automorphisms. 
Let $\Phi\in\mathcal{COC}_2(\widehat\Gamma)$ be the flip of $\Psi$ 
\[\Phi(\widehat\gamma_1,\widehat\gamma_2) = \Psi(\widehat\gamma_2,\widehat\gamma_1).\]  
Applying Rieffel deformation $\RD^\Phi$ to $(\widehat{A},\rho^{\widehat{A}})$  we get $\widehat\GG^\Phi = (\widehat{A}^\Phi,\widehat{\Delta}^\Phi)$. It turns out that  certain aspects of $\widehat\GG$ - $\GG$ duality have its counterpart on the  $\widehat\GG^\Phi$ - $\GG^\Psi$ level. In order to explain this  let us recall that multiplicative unitary $W\in\B(H\otimes H)$ may be viewed as the bicharacter  $W\in\M(\widehat A\otimes A)$:
\[\begin{split}
(\id\otimes\Delta)W &= W_{12}W_{13}\\
(\widehat\Delta\otimes\id)W &= W_{23}W_{13}
\end{split}
\]
Moreover  slicing $W$ with the normal functionals  we recover $A$ and $\widehat A$:
\begin{equation}\label{slices}\begin{split}
A&= \{(\omega\otimes\id)W:\omega\in\B(H)_*\}^{\textrm{cls}}\\
 \widehat A&= \{(\id\otimes\omega)W:\omega\in\B(H)_*\}^{\textrm{cls}}
\end{split}
\end{equation}
In what follows we shall construct certain unitary element $W^\Psi$ satisfying the braided counterpart of the bicharacter identity and   of the slice equality \eqref{slices}. Noting that 
\[\begin{split}\Psi&\in\M(\C^*(\Gamma)\otimes\C^*(\Gamma))\subset\M(\Gamma\ltimes\widehat{A}\otimes\Gamma\ltimes A)\\
W&\in\M(\widehat{A}\otimes A)\subset\M(\Gamma\ltimes\widehat{A}\otimes\Gamma\ltimes A)
\end{split}\] we define 
\begin{equation}\label{wpsi}
W^\Psi = \Psi W\Psi^*\in\M(\Gamma\ltimes\widehat{A}\otimes\Gamma\ltimes A)\end{equation}

Representing  $\Gamma\ltimes A$ and  $\Gamma\ltimes \widehat{A}$  on $L = \Ltwo(\Gamma)\otimes \Ltwo(\GG)$ as described in Remark \ref{repcp} we may view $W^\Psi$ as an operator on $L\otimes L$. The algebra of compact operators on $L$ is denoted by $\mathcal{K}(L)$. 
\begin{theorem}\label{bichar}
Adopting the above notation we have
\begin{equation}\label{slice1}
A^\Psi = \{(\omega\otimes \id)(W^\Psi):\omega\in\B(L)_*\}^{\textrm{cls}}.\end{equation}
\end{theorem}
\begin{proof} We begin by checking   Landstad conditions \eqref{lancod} for $(\omega\otimes \id)(W^\Psi)$. 
In order to check  $(\id\otimes\widehat\rho^\Psi)$
- invariance of $W^\Psi$ let us note that 
\[\begin{split}
(\id\otimes\rho^\Psi_{\widehat\gamma})(\Psi) &= \Psi(\lambda_{\tilde\Psi(\widehat\gamma)}\otimes\I)\\
(\id\otimes\rho^\Psi_{\widehat\gamma})(W) & = (\I\otimes \lambda_{\tilde\Psi(\widehat\gamma)})^*W(\I\otimes \lambda_{\tilde\Psi(\widehat\gamma)})\\
(\I\otimes \lambda_{\tilde\Psi(\widehat\gamma)})^*W(\I\otimes \lambda_{\tilde\Psi(\widehat\gamma)})&= (\lambda_{\tilde\Psi(\widehat\gamma)}\otimes \I)^*W(\lambda_{\tilde\Psi(\widehat\gamma)}\otimes \I)
\end{split}\] 
where the last equality follows from $\rho^A$ - $\rho^{\widehat{A}}$ relation  \eqref{dualityrel}
\[(\id\otimes\rho^A_\gamma) W = (\rho^{\widehat A}_\gamma\otimes\id) W.\]  
We compute
\[\begin{split}
(\id\otimes\widehat\rho^\Psi_{\widehat\gamma})(\Psi W\Psi^*) &=  \Psi(\lambda_{\tilde\Psi(\widehat{\gamma})}\otimes \id)(\id\otimes \lambda_{\tilde\Psi(\widehat{\gamma})})^*W(\lambda_{\tilde\Psi(\widehat{\gamma})}\otimes \id)^*(\id\otimes \lambda_{\tilde\Psi(\widehat{\gamma})})\Psi^*\\& = 
\Psi(\rho^{\widehat{A}}_{\tilde\Psi(\widehat\gamma)}\otimes\rho^A_{-\tilde\Psi(\widehat\gamma)})(W)\Psi^*=W^\Psi
\end{split}\] 

Checking that the map $\Gamma\ni\gamma\mapsto\lambda_\gamma(\omega\otimes\id)(W^\Psi)\lambda_\gamma^*$ is norm continuous (which is the second Landstad condition)  boils down to the simple identity 
\[ \lambda_\gamma(\omega\otimes\id)(W^\Psi)\lambda_\gamma^*= (\lambda_\gamma^*\cdot\omega\cdot \lambda_\gamma\otimes\id)(W^\Psi) \] and the fact that the map $\Gamma\ni\gamma\mapsto\lambda_\gamma^*\cdot\omega\cdot \lambda_\gamma \in\B(\Ltwo(L))_* $ is norm continuous. 

The third Landstad condition in our case has the form 
\[x\left((\omega\otimes\id)(W^\Psi)\right)y\in \Gamma\ltimes A \textrm{ for } x,y\in\C^*(\Gamma) \textrm{ and } \omega\in\B(\Ltwo(L))_*\] In order to check it we are reasoning as in  the proof of Theorem 4.4 of \cite{Kasp}. Actually the   argument   shows that \[A^\Psi = \{(\omega\otimes \id)(W^\Psi):\omega\in\B(\Ltwo(L))_*\}^{\textrm{cls}}.\] Let us  first define  
\[\begin{split} 
\mathcal{V} & =  \{x_1\cdot (\omega\otimes \id)(W^\Psi)\cdot x_2:\omega\in\B(\Ltwo(L))_*,\,\,x_1,x_2\in\C^*(\Gamma)\}^{\textrm{cls}}
\end{split}\]
and  note that 
\[
\begin{split}
\mathcal{V} & =  \{x_1\cdot (y_2\cdot\omega\cdot y_1\otimes \id)(W^\Psi)\cdot x_2:\omega\in\B(\Ltwo(L))_*,\,\,x_1,x_2,y_1,y_2\in\C^*(\Gamma)\}^{\textrm{cls}}\\& =  \{(\omega \otimes \id)((y_1\otimes x_1)\Psi W \Psi^* (y_2\otimes x_2)):\omega\in\B(\Ltwo(L))_*,\,\,x_1,x_2,y_1,y_2\in\C^*(\Gamma)\}^{\textrm{cls}}\\& =\{(\omega \otimes \id)((y_1\otimes x_1)W  (y_2\otimes x_2)):\omega\in\B(\Ltwo(L))_*,\,\,x_1,x_2,y_1,y_2\in\C^*(\Gamma)\}^{\textrm{cls}}\\& =\overline{\C^*(\Gamma)A\C^*(\Gamma)}^{\|\cdot\|}= \Gamma\ltimes A
\end{split} \] 
Using Lemma 2.6 of \cite{Kasp} we conclude \eqref{slice1}.
\end{proof}
The embeddings of $A^\Psi$ into the first and the second leg of  $A^\Psi\otimes_\Psi A^\Psi$ are denoted $\iota_1^{A^\Psi},\iota_2^{A^\Psi}$ respectively:
\[\begin{split}\iota_1^{A^\Psi}(a) &= a\otimes_\Psi\I\\
\iota_2^{A^\Psi}(a) &= \I\otimes_\Psi a
\end{split}
\] for any $a\in A^\Psi$. The next theorem may be considered as a step towards the  counterpart of the bicharacter equation satisfied by $W^\Psi$ which is expected to have the form 
\[ (\id\otimes\Delta^\Psi)(W^\Psi) =(\id\otimes \iota^{A^\Psi}_1)(W^\Psi)(\id\otimes \iota^{A^\Psi}_2)(W^\Psi)\] Since the equality \eqref{slice1} does not automatically imply that $W^\Psi\in\M(\mathcal{K}(L)\otimes A^\Psi)$ one should be careful writing $(\id\otimes\Delta^\Psi)(W^\Psi)$. We circumvent this using the defining equality $\Delta^\Psi =\lambda\ltimes\Delta|_{A^\Psi}$ where $\lambda\ltimes\Delta\in\Mor(\Gamma\ltimes A,\Gamma\ltimes(A\otimes A))$ is the extension of $\Delta\in\Mor_{\Gamma}(A,A\otimes A)$ to the crossed product $\C^*$ - algebras. Similarly  since the canonical embeddings $\iota^{A}_1,\iota^{A}_2\in\Mor_\Gamma(A,A\otimes A)$ are covariant we may extend them to the corresponding crossed product $\C^*$ - algebras which we denote by $\lambda\ltimes\iota^A_1,\lambda\ltimes\iota^A_2\in\Mor(\Gamma\ltimes A,\Gamma\ltimes (A\otimes A))$ respectively and one has \[\begin{split}\iota_1^{A^\Psi} &= \lambda\ltimes\iota^A_1|_{A^\Psi}\\ 
\iota_2^{A^\Psi} &= \lambda\ltimes\iota^A_2|_{A^\Psi}
\end{split}\]
\begin{theorem}
Let $W^\Psi$ be the operator defined in \eqref{wpsi}. Then
\begin{equation}\label{delw}(\id\otimes(\lambda\ltimes\Delta))(W^\Psi) =(\id\otimes\lambda\ltimes \iota^{A}_1)(W^\Psi)(\id\otimes\lambda\ltimes \iota^{A}_2)(W^\Psi)\end{equation}
In particular if $W^\Psi\in\M(\mathcal{K}(L)\otimes A^\Psi)$ then 
\begin{equation}\label{addassum}(\id\otimes\Delta^\Psi)(W^\Psi) =(\id\otimes \iota^{A^\Psi}_1)(W^\Psi)(\id\otimes \iota^{A^\Psi}_2)(W^\Psi)\end{equation}
\end{theorem}
\begin{proof}
In order to prove \eqref{delw} we compute
\[\begin{split} 
 (\id\otimes\lambda\ltimes\Delta)(\Psi W \Psi^*)&=
(\id\otimes\lambda\ltimes\Delta)(\Psi) \left((\id\otimes\Delta)W \right) (\id\otimes\lambda\ltimes\Delta)(\Psi)^*\\
&= (\id\otimes\lambda\ltimes\Delta)(\Psi) W_{12}W_{13}(\id\otimes\lambda\ltimes\Delta)(\Psi)^*\\&=
(\id\otimes\Gamma\ltimes\iota^{A}_1)(\Psi)W_{12}W_{13}(\id\otimes\Gamma\ltimes\iota^{A}_2)(\Psi)
\end{split}\] In the last step  we used the equality of three respective restrictions \[\lambda\ltimes\Delta|_{\C^*(\Gamma)} = \Gamma\ltimes\iota^A_1|_{\C^*(\Gamma)} = \Gamma\ltimes\iota^A_2|_{\C^*(\Gamma)}\]  all of them coinciding with the standard embedding  $\C^*(\Gamma)\subset\M(\Gamma\ltimes (A\otimes A))$. Noting that \[\begin{split}(\id\otimes\Gamma\ltimes\iota^A_1)(W) &= W_{12}\\ (\id\otimes \Gamma\ltimes\iota^A_2)(W)& = W_{13}\end{split}\] we get 
\[\begin{split} 
(\id\otimes\lambda\ltimes\Delta)(\Psi W \Psi^*)  &=
(\id\otimes\lambda\ltimes \iota^{A}_1)(\Psi W\Psi^*)(\id\otimes\lambda\ltimes \iota^{A}_2)(\Psi W\Psi^*)
\end{split}\]
Since $\lambda\ltimes\Delta|_{A^\Psi} = \Delta^\Psi$, $\lambda\ltimes \iota^{A}_1|_{A^\Psi} = \iota_1^{A^\Psi}$ and $\lambda\ltimes \iota^{A}_2|_{A^\Psi} = \iota_2^{A^\Psi}$ then with the assumption that $W^\Psi\in\M(\mathcal{K}(L)\otimes A^\Psi)$ we get \eqref{addassum}.
\end{proof}
Since the multiplicative unitary $\widehat W$ for $\widehat\GG$ is given by $\Sigma W^*\Sigma $ we immediately get the analogous results for 
$(\widehat{A}^\Phi,\widehat\Delta^\Phi)$. 
In particular \[
\widehat A^\Phi = \{(\id\otimes \omega)W^\Psi:\omega\in\B(\Ltwo(L))_*\}^{\textrm{cls}}.\]
Actually one expects to have  $W^\Psi\in\M(\widehat A^\Phi\otimes A^\Psi)$ which we were not able to prove even under the assumption that    $W^\Psi\in\M(\mathcal{K}(L)\otimes  A^\Psi)$ and $W^\Psi\in\M(\widehat A^\Phi\otimes\mathcal{K}(L))$ . In the quantum group theory the respective result is obtained by the pentagonal equation which is not available for $W^\Psi$. 

\section{On a certain  example of a quantum Minkowski space}
In this section we  employ  Rieffel deformation to  get an example of a $\C^*$-quantum Minkowski space. In order to do it we  elaborate the example described in \cite{kasp1} where   Rieffel deformation functor $\RD^\Psi$ was applied to the $\C^*$-algebra $\C_0(\mathcal{M})$ of the Minkowski space $\mathcal{M}$. Let us first give the concise description of the results  concerning the example of a quantum Minkowski space described in  \cite{kasp1}.  Since in the context of this example the Minkowski space was  equipped with the action of $SL(2,\mathbb{C})$ (the twofold cover of the connected component of the Lorentz group) we identify $\mathcal{M}$ with the space of selfadjoint matrices $\mathcal{H}$  
\begin{equation}\label{clgen}\mathcal{H}=\left\{\begin{bmatrix} x&w\\\bar{w}&y\end{bmatrix}:\,x,y\in\mathbb{R},\,w\in\mathbb{C}\right\}\end{equation} 
where the identifiaction is given by the map \[\mathcal{M}\ni[x_\mu] = \begin{bmatrix}x_0\\x_1\\x_2\\x_3\end{bmatrix} \mapsto\sigma([x_\mu]) = \begin{bmatrix} x_0+x_3&x_1+\imath x_2\\x_1-\imath x_2&x_0-x_3\end{bmatrix}\in\mathcal{H}.\]
With this identification the action of $SL(2,\mathbb{C})$ on $\mathcal{M}$ is given by matrix conjugation
\[\sigma(\alpha_g[x_\mu]) = g\sigma([x_\mu])g^*\] Equipping $\mathcal{M}$ with the addition
\[[x_\mu]+[y_\mu] = [x_\mu+y_\mu]\] we view   $\alpha_g$ as    automorphism  of $\mathcal{M}$.

On the  $\C^*$-level  we introduce a notation  $\mathbb{M} =(\C_0(\mathcal{M}),\Delta)$ where \[\Delta(f)([x_\mu],[y_\mu]) = f([x_\mu+y_\mu])\]  for any  $f\in\C_0(\mathcal{M})$.  
Using the restriction of $\alpha$ to the  subgroup of diagonal matrices \[  \left\{\begin{bmatrix} a& 0\\0& a^{-1}\end{bmatrix}:a\in\mathbb{C}\right\}\subset SL(2,\mathbb{C})\] we get the data needed to  perform Rieffel deformation of $\mathbb{M}$. For convenience  we pull  back  the action to $\mathbb{C}$   by the group homomorphism:
\[\mathbb{C}\ni z\mapsto\begin{bmatrix} e^z& 0\\0& e^{-z}\end{bmatrix}\in\Gamma\] and we define $\rho:\mathbb{C}\rightarrow\Aut(\mathbb{M})$
\[(\rho_z f)([x_\mu]) = f(\alpha_{\sigma(z)}[x_\mu])\]
We identify  $\mathbb{C}$ and $\widehat{\mathbb{C}}$ by the duality 
\[\langle z_1,z_2\rangle =\exp(\imath  \Im(z_1 z_2))\]
and define a bicharacter  on $\widehat{\mathbb{C}}$: \[\Psi(z_1,z_2) = \exp(\imath s\Im(z_1\overline{z_2}))\]  $s\in\mathbb{R}$ being the deformation parameter. 

Let $\mathbb{M}^\Psi = (\C_0(\mathcal{M})^\Psi, \Delta^\Psi)$ be the Rieffel deformation of $\mathbb{M}$ as described in Section \ref{bqg}.  
In what follows we shall introduce the $\mathbb{M}^\Psi$-counterparts of   $x,y, w$ coordinates. In order to do this let us consider a unitary $V\in\M(\C^*(\mathbb{C}))$  such that $V(z) = \exp\left(\imath\frac{s}{2}\Im(z^2)\right)$  - note the $\C^*(\mathbb{C}) \cong\C_0(\mathbb{C})$ identification which we use to interpret $V$ as a function on $\mathbb{C}$.  We define   
\begin{equation}\label{defcoef} \widehat x=e^{-2s}V x V^*,\,\widehat y=e^{-2s}V y  V^*,\,\widehat w=e^{2s}V^* w V
\end{equation} which are  elements affiliated with a $\C^*$-algebra $\mathbb{C}\ltimes \C_0(\mathcal{M})$. 
For the concept of an element $T$ affiliated with a $\C^*$-algebra  $A$, $T\,\eta A$ we refer to \cite{woraffun} where an idea  of $\C^*$-algebras generated by affiliated elements is also developed.   In  Theorem 5.3 of \cite{kasp1} we proved  that $\widehat x,\widehat y,\widehat w\,\eta\C_0(\mathcal{M})^\Psi$ and $\C^*$-algebra $\C_0(\mathcal{M})^\Psi$ is generated by $\widehat x,\widehat y,\widehat w$. Remarkably $\widehat x,\widehat y,\widehat w$ satisfy the following commutation relations
\begin{equation}\label{comrel1}\begin{split}
\widehat{x}\widehat{w} &= t^{-1} \widehat{w} \widehat{x}\\
\widehat{x}\widehat{w}^* &= t  \widehat{w}^* \widehat{x}\\
\widehat{y}\widehat{w} &= t \widehat{w} \widehat{y}\\
\widehat{y}\widehat{w}^* &= t \widehat{w}^* \widehat{y}
\end{split}
\end{equation}
where $t = e^{-8s}$. For the precise  meaning of \eqref{comrel1} we advice to consult Theorem 3.7 of \cite{kasp1} where the explanation  of the relation 
\begin{align*} RS&=p^2SR\\
RS^*&=q^2S^*R
\end{align*}  for normal operators $R,S\,\eta A$ is  given. 

In the formulation of the next theorem, where we analyze the action of $\Delta^\Psi$ on generators, we use the extension of notation \eqref{sugnot} to affiliated elements writing   \[T\otimes_\Psi\I, \I\otimes_\Psi T\,\eta\,\, (A^\Psi\otimes_\Psi A^\Psi) \] for $T\,\eta\, A^\Psi$. 

\begin{theorem}\label{commrel}
Let $\widehat{x},\widehat{y},\widehat{w}\,\eta\,\C_0(\mathcal{M})^\Psi$ be the affiliated elements defined above. 
Then
\begin{itemize}
\item  $\left(\widehat{x}\otimes_\Psi\I,\I\otimes_\Psi\widehat{x}\right)$, $\left(\widehat{y}\otimes_\Psi\I,\I\otimes_\Psi\widehat{y}\right)$, $\left(\widehat{w}\otimes_\Psi\I,\I\otimes_\Psi\widehat{w}\right)$ and $(\widehat{x}\otimes_\Psi\I, \I \otimes_\Psi\widehat{y})$  are strongly commuting pairs of normal elements affiliated with $\C_0(\mathcal{M})^\Psi\otimes_\Psi\C_0(\mathcal{M})^\Psi$ and we have
\begin{equation}\label{delpcor}\begin{split}
\Delta^\Psi_{\mathcal{M}}(\widehat{x}) &= \widehat{x}\otimes_\Psi\I+\I\otimes_\Psi\widehat{x}\\
\Delta^\Psi_{\mathcal{M}}(\widehat{y}) &= \widehat{y}\otimes_\Psi\I+\I\otimes_\Psi\widehat{y}\\
\Delta^\Psi_{\mathcal{M}}(\widehat{x}) &= \widehat{w}\otimes_\Psi\I+\I\otimes_\Psi\widehat{w}
\end{split}
\end{equation}

\item  $\left(\widehat{x}\otimes_\Psi\I,\I\otimes_\Psi\widehat{w}\right)$  and $\left(\widehat{y}\otimes_\Psi\I,\I\otimes_\Psi\widehat{w}\right)$ are respectively   $(t^{-1},t)$ and $(t,t^{-1})$-commuting pair of normal elements affiliated with $\C_0(\mathcal{M})^\Psi\otimes_\Psi\C_0(\mathcal{M})^\Psi$.
\end{itemize}
\end{theorem}
\begin{proof}Note that
 \[\begin{split}
 \widehat{x}\otimes_\Psi\I&= e^{-2s}U(x\otimes \I)U^*\\
\I\otimes_\Psi\widehat{x}&= e^{-2s}U(\I\otimes x)U^*\\
  \widehat{y}\otimes_\Psi\I&=e^{-2s} U(y\otimes \I)U^*\\
\I\otimes_\Psi\widehat{y}&= e^{-2s}U(\I\otimes y)U^*\\
  \widehat{w}\otimes_\Psi\I&= e^{ 2s}U(w\otimes \I)U^*\\
\I\otimes_\Psi\widehat{w}&= e^{ 2s}U(\I\otimes w)U^*\\
\end{split}\]
where   $U$ is the image of $V = \exp\left(\imath\frac{s}{2}\Im(z^2)\right)\in\M(\C^*(\mathbb{C}))$ under the embedding \[\iota^{\C^*(\Gamma)}\in\Mor(\C^*(\mathbb{C}), \mathbb{C}\ltimes(\C_0(\mathcal{M})\otimes \C_0(\mathcal{M}))).\]
 Noting that 
 \begin{itemize}
 \item the commutation relation between    $ U $ and $x\otimes 1, y\otimes \I $ and $w\otimes \I$  are the same as in the case of $V$ and $x,y,w$
 \item  the commutation relation between    $ U $ and $  \I\otimes x,  \I\otimes y $ and $ \I\otimes w$  are the same as in the case of $V$ and $x,y,w$
 \end{itemize}
 the techniques of the proof of Theorem  5.3 of \cite{kasp1} may   be applied giving the commutation relation of our theorem.
 
Let us compute   $\Delta^\Psi(\widehat{x})$: 
 \[\begin{split}\Delta^\Psi(\widehat{x}) &= (\lambda\ltimes \Delta)(e^{-2s}VxV^*) \\&= e^{-2s}U(x\otimes 1 + 1\otimes x)U^*\\&=e^{-2s}U(x\otimes 1)U^* + e^{-2s}U(1\otimes x)U^*\\& = \widehat{x}\otimes_\Psi\I + \I\otimes_\Psi\widehat{x}
 \end{split}
 \] Similarly we get the remaining identities of \eqref{delpcor}.  
\end{proof}

\section*{Acknowledgements}
The author  wish to thank S.L. Woronowicz, S. Roy and P. So\l tan for helpful  discussions. 


\begin{thebibliography}{66}
\bibitem{BrOz} N.P.~Brown \& N.~Ozawa: \emph{$\C^*$-Algebras and Finite-Dimensional Approximations}, Graduate Studies in Mathematics \textbf{88} (2008). 
\bibitem{DKSS} M. Daws, P. Kasprzak, A. Skalski,  P.M. So\l tan: Closed quantum subgroups of locally compact quantum groups \emph{ Adv. Math.} \textbf{231} (2012) 3473-3501.
\bibitem{Kasp}
P.~Kasprzak: Rieffel deformation via crossed products \emph{J. Funct. Anal.} \textbf{257} (2009), 1288--1332.
\bibitem{kasp1}
P.~Kasprzak: Rieffel deformation of group coactions \emph{Commun. Math. Phys.} \textbf{300} (2010), 741--763.

\bibitem{KV}
J.~Kustermans \& S.~Vaes: Locally compact quantum groups. \emph{Ann.~Scient.~\'{E}c.~Norm.~Sup.} $4^{\text{\tiny e}}$ s\'{e}rie, t.~\textbf{33} (2000), 837--934.
\bibitem{lan} M.B.~Landstad: Duality theory for covariant systems \emph{Trans.~AMS} \textbf{248} (1979), no.~2, p. 223 -- 267.
\bibitem{Majid} S.~Majid \emph{Foundations of quantum group theory.} Cambridge University Press 1995.
\bibitem{mnw}
T.~Masuda, Y.~Nakagami \& S.L.~Woronowicz: A $\mathrm{C}^*$-algebraic framework for the quantum groups. \emph{Int.~J.~Math.} \textbf{14} (2003), 903--1001.
\bibitem{MRW} R.~Meyer, S.~Roy and S.~L.~Woronowicz: Homomorphisms of quantum groups \emph{Muenster J. Math} \textbf{5}, no. 1 (2012),  1--24.
\bibitem{MRW1} R.~Meyer, S.~Roy and S.~L.~Woronowicz:
Quantum group - twisted tensor products of C*-algebras, \emph{Inter. J. Math.} \textbf{25}, no. 2 (2014).
\bibitem{Rf1} M.A.~Rieffel: Deformation quantization for action of
$\mathbb{R}^d$ \emph{ Mem.~Am.~Math.~Soc.} \textbf{506} (1993).
\bibitem{Roy} S.~Roy: Quantums groups with projection. Ph.D. dissertation, George-August-Universit\"at G\"ottingen, 2013.
\bibitem{W5sol}P.M.~So\l tan, S.~L.~Woronowicz: A remark on manageable multiplicative unitaries. \emph{Lett. Math. Phys.} \textbf{57} (2001), 239--252.
\bibitem{SolWor}P.~So\l tan, S.~L.~Woronowicz: From multiplicative unitaries to quantum groups II \emph{  J. Funct. Anal.} \textbf{252} (1) (2007), 42--67.

\bibitem{Wil} D.~Wiliams: \emph{Crossed Products of $\C^*$-algebras} Mathematical Surveys and Monographs, \textbf{134}, (2007). 
\bibitem{W3}
  S.L.~Woronowicz: Operator Equalities Related to the
Quantum $E(2)$ Group \emph{ Commun.~in Math.~Phys.} \textbf{144} (1992), no~2, 417--428.
 \bibitem{woraffun} S.~L.~Woronowicz: $\C^*$-algebras generated by unbounded elements \emph{ Rev. Math.l Phys.}, \textbf{ 7}, no.~3,  (1995), 481--521.
 \bibitem{W5}   S.L.~Woronowicz: From multiplicative unitaries to quantum groups \emph{International Journal of Math.} \textbf{7}, No.~1 (1996).


\end{thebibliography}
\end{document}